\documentclass[letterpaper]{amsart}

\usepackage[letterpaper,left=1.5in,right=1.5in,top=1in,bottom=1in]{geometry}

\newcommand{\myauthor}{Benjamin Antieau and Ben Williams}
\newcommand{\mytitle}{Unramified division algebras do not always contain Azumaya maximal orders}

\title{\mytitle}
\author{\myauthor}

\usepackage[pdfstartview=FitH,
            pdfauthor={\myauthor},
            pdftitle={\mytitle},
            colorlinks,
            linkcolor=reference,
            citecolor=citation,
            urlcolor=e-mail,
            ]{hyperref}

\usepackage{amsmath}
\usepackage{amscd}
\usepackage{amsbsy}
\usepackage{amssymb}
\usepackage{verbatim}
\usepackage{eufrak}
\usepackage{microtype}
\usepackage{hyperref}
\usepackage{mathrsfs}
\usepackage{amsthm}
\usepackage[all,cmtip]{xy}
\usepackage{tikz}
\usetikzlibrary{matrix,arrows}
\usepackage[abbrev,lite]{amsrefs}

\usepackage{mathpazo}

\usepackage{color}
\definecolor{todo}{rgb}{1,0,0}
\definecolor{conditional}{rgb}{0,1,0}
\definecolor{e-mail}{rgb}{0,.40,.80}
\definecolor{reference}{rgb}{.20,.60,.22}
\definecolor{mrnumber}{rgb}{.80,.40,0}
\definecolor{citation}{rgb}{0,.40,.80}

\DeclareMathAlphabet{\mathpzc}{OT1}{pzc}{m}{it}
\usepackage{dsfont}

\DeclareMathOperator{\PGL}{PGL}

\DeclareMathOperator{\PU}{PU}
\DeclareMathOperator{\SL}{SL}

\DeclareMathOperator{\tors}{tors}

\DeclareMathOperator{\Spec}{Spec}

\DeclareMathOperator{\cl}{cl}

\DeclareMathOperator{\Hoh}{H}

\DeclareMathOperator{\Hom}{Hom}
\DeclareMathOperator{\Br}{Br}

\DeclareMathOperator{\per}{per}

\DeclareMathOperator{\ind}{ind}

\renewcommand{\P}{\operatorname{P}}

\newcommand{\et}{\mathrm{\acute{e}t}}

\newcommand{\weq}{\simeq}

\newcommand{\B}{\mathrm{B}}

\DeclareFontEncoding{OT2}{}{}

\renewcommand{\cl}{\mathrm{cl}}

\newcommand{\Mod}{\mathrm{Mod}}

\newcommand{\we}{\simeq}
\newcommand{\iso}{\cong}

\newcommand{\A}{\mathds{A}}

\newcommand{\CC}{\mathds{C}}
\newcommand{\RR}{\mathds{R}}

\newcommand{\ZZ}{\mathds{Z}}

\newcommand{\Gm}{\mathds{G}_{m}}

\newcommand{\RP}{{\RR \mathrm{P}}}

\newcommand{\Ascr}{\mathscr{A}}
\newcommand{\Bscr}{\mathscr{B}}

\let\oldmarginpar\marginpar
\renewcommand\marginpar[1]{\-\oldmarginpar[\raggedleft\footnotesize #1]%
{\raggedright\footnotesize #1}}

\newcommand{\tensor}{\otimes}

\newcommand{\Brm}{\mathrm{B}}

\newcommand{\BPU}{{\Brm \PU}}

\newcommand{\SU}{\mathrm{SU}}

\newcommand{\BP}{\mathrm{BP}}
\newcommand{\BPGL}{\mathrm{BPGL}}

\newcommand{\Mrm}{\mathrm{M}}

\theoremstyle{plain}
\newtheorem{theorem}{Theorem}[section]
\newtheorem{lemma}[theorem]{Lemma}

\newtheorem{proposition}[theorem]{Proposition}

\newtheorem{corollary}[theorem]{Corollary}

\newtheorem{problem}[theorem]{Problem}

\newtheoremstyle{named}{}{}{\itshape}{}{\bfseries}{.}{.5em}{#1 \thmnote{#3}}
\theoremstyle{named}

\theoremstyle{definition}

\newtheorem{example}[theorem]{Example}

\theoremstyle{remark}

\begin{document}

\begin{abstract}
    We show that, in general, over a regular integral noetherian affine scheme $X$ of
    dimension at least $6$, there exist Brauer
    classes on $X$ for which the associated division algebras over the generic point have no Azumaya
    maximal orders over $X$. Despite the algebraic nature of the
    result, our proof relies on the topology of classifying spaces of algebraic groups.
\end{abstract}

\maketitle

\section{Introduction}

Let $K$ be a field. The Artin--Wedderburn Theorem implies that every central simple
$K$-algebra $A$ is isomorphic to an algebra $\Mrm_n(D)$ of $n\times
n$ matrices over a finite dimensional central $K$-division algebra $D$. One says $\Mrm_n(D)$ and $\Mrm_{n'}(D')$ are Brauer-equivalent when $D$
and $D'$ are isomorphic over $K$. The set of Brauer-equivalence classes forms a group under tensor product, $\Br(K)$, the Brauer group of $K$.
The \textit{index} of an equivalence class $\alpha=\cl(\Mrm_{n}(D))\in\Br(K)$ is the degree of the minimal representative, $D$ itself.

Let $X$ be a connected scheme. The notion of a central simple algebra over a
field was generalized by Auslander--Goldman
\cite{auslander-goldman} and by Grothendieck~\cite{grothendieck-brauer-2} to the concept of an Azumaya algebra over $X$. An Azumaya
algebra $\Ascr$ is a locally-free sheaf of algebras which \'etale-locally takes the form of
a matrix algebra. That is, there is an \'etale cover $\pi:U\rightarrow X$ such that
$\pi^*\Ascr\iso\Mrm_n(\mathscr{O}_U)$. In this case, the \textit{degree} of $\Ascr$ is $n$. Brauer equivalence and
a contravariant Brauer group functor may be defined in this context, generalizing the
definition of the Brauer group over a
field. Over a scheme, we do not have Artin--Wedderburn theory and consequently we cannot be certain without further work that a
Brauer class $\alpha\in\Br(X)$ contains an Azumaya algebra $\Ascr$ whose degree divides that of all other Azumaya algebras having class
$\alpha$. The index of $\alpha$ is therefore defined to be the greatest common divisor of the degrees of Azumaya algebras with Brauer class
$\alpha$, rather than the minimum among such degrees.

Let $X$ be a regular integral noetherian scheme with generic point $\Spec K$. Given a central simple $K$-algebra $A$, an
\textit{order} in $A$ over $X$ is a torsion-free coherent $\mathscr{O}_X$--algebra, $\Ascr$, such that $\Ascr \otimes_{\mathscr{O}_X} K \iso
A$. A \textit{maximal order} in $A$ over $X$ is an order which is not a proper subalgebra of any other
order in $A$ over $X$.

If $\Ascr$ is an Azumaya algebra on $X$ that restricts to $A$ over the generic point, the class $\alpha$ of $A$ must be in the image of the
map $\Br(X)\rightarrow\Br(K)$, a map which is known to be injective by~\cite{grothendieck-brauer-2}*{Corollaire 1.10}. In this paper, we
shall consider only classes $\alpha\in\Br(X)\subseteq\Br(K)$; these are said to be \textit{unramified along $X$\/}. An Azumaya algebra
$\Ascr$ that restricts to $A$ is a maximal order in $A$.

We ask the following questions. First, does a class $\alpha \in \Br(X)$ of index $d$ necessarily contain an Azumaya algebra $\Ascr$ of
degree $d$? Second, do all division algebras $A$ over $K$ for which $\alpha=\cl(A)$ is unramified along $X$ contain an Azumaya maximal order
$\Ascr$ over $X$? These questions are equivalent: if $A$ is an unramified division algebra, then any Azumaya maximal order in $A$ over $X$
has degree dividing the degrees of all other Azumaya algebras with class $\alpha$. On the other hand, under our assumption that $X$ is
regular and noetherian, the index $\ind(\alpha)$ can be computed either on $X$ or over $K$, by~\cite{aw3}*{Proposition 6.1}. If
$\ind(\alpha)=d$ and if $\Ascr$ is an Azumaya algebra of degree $d$ over $X$ with class $\alpha$, then it follows that $\Ascr$ is an Azumaya
maximal order in the unique division algebra $A$ with class $\alpha$ over $K$.

The questions were answered in the affirmative by Auslander and Goldman~\cite{auslander-goldman} when $X$ is a regular noetherian affine
scheme of Krull dimension at most $2$. They argue, by~\cite{auslander-goldman}*{Theorem 2.1}, that when $\alpha$ is unramified, a maximal
order is Azumaya if and only if it is locally free as an $\mathscr{O}_X$-module.  Since $X$ is noetherian, every order is contained in a
maximal order, maximal orders are necessarily reflexive $\mathscr{O}_X$--modules, and reflexive $\mathscr{O}_X$--modules are locally free
outside a closed subset of codimension no less than $3$.

After the proof of Proposition 7.4 of~\cite{auslander-goldman}, they write:
\begin{quote} \small
    It should be remarked that the condition [$\dim R\leq 2$] on the dimension of $R$ was used in the proof
    only to ensure that $\Delta$ contains a maximal order which is $R$-projective. It is not
    known at the present time whether the restriction on the dimension of $R$ is actually
    necessary.
\end{quote}
The existence of Azumaya maximal orders when $\alpha$ is unramified has implications in the study of low-dimensional schemes. For example, it was used
in~\cite{auslander-goldman}*{Section 7} to prove that the Brauer group of a regular integral noetherian commutative ring $R$ can be
identified as the intersection
\begin{equation*}
    \Br(R)=\bigcap_{p\in\Spec R^{(1)}}\Br(R_p),
\end{equation*}
ranging over all primes of height $1$ in $R$. This result was later extended to regular schemes of finite type over a
field by Hoobler~\cite{hoobler} using \'etale cohomology.

A second application of the low-dimensional result of Auslander and Goldman is found in the proof of de Jong~\cite{dejong} that
$\per(\alpha)=\ind(\alpha)$ for $\alpha\in\Br(k(X))$, the Brauer group of the function field of a surface over an algebraically closed
field $k$. De Jong's proof shows that one can reduce to the case where $\alpha\in\Br(X)$, where $X$ is a smooth projective model for $k(X)$, and
then one constructs an Azumaya algebra of degree equal to $\per(\alpha)$ and with Brauer
class $\alpha$. The reduction to the unramified case holds in all
dimensions, by de Jong and Starr~\cite{dejong-starr}. One is therefore naturally lead to ask whether an analogue to the second part might be
practicable in higher dimensions: if $X$ is smooth and projective and $\alpha\in\Br(k(X))$ is unramified along $X$, does there exist an
Azumaya algebra on $X$ with class $\alpha$ and degree equal to the index of $\alpha$?

We show that the restriction $\dim X\leq 2$ in the result of Auslander and Goldman is necessary, and that the answer to either of our equivalent questions is negative even when $X$ is
a connected smooth affine complex variety. We shall require our varieties to be irreducible in the sequel.

\begin{theorem}\label{thm:main}
  Let $n>1$ be an odd integer. There exists a smooth affine variety $X$ of dimension $6$ over the complex numbers and an Azumaya algebra
  $\Ascr$ of degree $2n$ and period $2$ such that there is no degree-$2$ Azumaya algebra
  with class $\cl(\Ascr)\in\Br(X)$.
\end{theorem}

Examples as in the theorem exist in all dimensions at least equal to $6$; it is unknown whether they exist when $3\leq\dim X\leq 5$. Our
result contrasts with \cite{demeyer}*{Corollary 1}, where a Wedderburn-type theorem is shown to hold for semi-local rings having only
trivial idempotents.

\begin{corollary}
    There exists a smooth affine variety $X$ of dimension $6$ over the complex numbers and Brauer
    classes $\alpha\in\Br(X)$ such that the division algebra over the generic point has no
    Azumaya maximal order over $X$.
\end{corollary}
\begin{proof}
  Fix an odd integer $n>1$. Take $X$ and $\alpha=\cl(\Ascr)$ as provided by the theorem. Since $\per(\alpha)=2$, and
  $\ind(\alpha)$ is a power of $2$ that divides $2n$, over the generic point $\eta$, there is a degree-$2$ division algebra $A$ with
  class $\alpha$.  There is no degree-$2$ Azumaya algebra of class $\alpha$ over $X$, by the
  theorem, and so no maximal order in $A$ is Azumaya.
\end{proof}

The next corollary shows that, in general, there is no prime decomposition for Azumaya algebras as there
is for central simple algebras. This answers a question of Saltman~\cite{saltman}, appearing after Theorem
5.7.
\begin{corollary}
    For $n>1$ odd, there is a smooth affine complex variety $X$ and an Azumaya algebra $\Ascr$ on $X$ of degree
    $2n$ and period $2$ such that $\Ascr$ has no decomposition
    $\Ascr\iso\Ascr_2\otimes\Ascr_n$ for Azumaya algebras of degrees $2$ and $3$,
    respectively.
\end{corollary}
 \begin{proof}
      One can again take $X$ and $\Ascr$ as in the proof of the theorem.
      We should have $\per(\Ascr) = \per(\Ascr_2) \per (\Ascr_n)$. In particular, there should exist a degree-$2$,
      period-$2$ Azumaya algebra over $X$ with class $\cl(\Ascr)$, contradicting our choice of $X$.
\end{proof}

The construction of our counterexamples uses algebraic topology and topological Azumaya algebras, studied in~\cite{aw1}. By examining the
topology of certain classifying spaces, we are able to prove non-existence results about topological Azumaya algebras for CW complexes. We
then pass to algebraic examples by using Totaro's algebraic approximations to classifying space of affine algebraic
groups,~\cite{totaro}. This yields smooth quasi-projective varieties over $\CC$. By using Jouanolou's device, we can replace these by smooth
affine varieties and by using the affine Lefschetz hyperplane theorem, we can fashion an example on a smooth affine $6$-fold.

Another closely related line of inquiry has been taken up by various authors starting with
DeMeyer~\cite{demeyer}. We recall Problem 5 from the book of DeMeyer and
Ingraham~\cite{demeyer-ingraham}*{Section V.3}.

\begin{problem}
    For which commutative rings $R$ is the following true? If $\Ascr$ is an Azumaya
    $R$-algebra,
    then there exists a unique Azumaya $R$-algebra $\mathscr{D}$ equivalent to $\Ascr$ in $\Br(R)$ such
    that $\mathscr{D}$ has no idempotents besides $0$ and $1$ and such that
    $\Ascr^{\mathrm{op}}\iso{_{\mathscr{D}}\Hom}(M,M)$
    for some projective left $\mathscr{D}$-module $M$ that generates $_{\mathscr{D}}\Mod$.
\end{problem}

The Wedderburn theorem says that fields have this property, and we therefore call it the Wedderburn
property. When $R$ has no idempotents besides $0$ and $1$,
there is always at least one $\mathscr{D}$ satisfying the condition of the problem, so the
question is the uniqueness of $\mathscr{D}$.
DeMeyer~\cite{demeyer} showed that all semi-local rings have the Wedderburn property. Examples appear
in Bass~\cite{bass-stable}*{Page 46} and Childs~\cite{childs} of number rings that do not
have the Wedderburn property. In these cases, because the
dimension of the ring is $1$, there is always a maximal order in the division algebra over
the field of fractions $K$; in the rings of Bass and Childs this
maximal order is not unique. Our theorem furnishes a different kind of example where the uniqueness fails.

\begin{example}
    Let $n$, $X=\Spec R$, and $\Ascr$ be as in Theorem~\ref{thm:main}. Then, the index of
    $\cl(\Ascr)$ is $2$. Therefore, by~\cite{aw3}*{Proposition 6.1}, there exist Azumaya algebras $\Bscr_1,\ldots,\Bscr_k$ on $X$ such
    that $\gcd(\deg(\Ascr),\deg(\Bscr_1),\ldots,\deg(\Bscr_k))=2$ and
    $\cl(\Bscr_i)=\cl(\Ascr)$ for $1\leq i\leq k$. By Morita theory,
    there is an equivalence $_{\Ascr}\Mod\we{_{\Bscr_i}\Mod}$ of abelian categories of left
    modules. Our example ensures that we can choose both $\Ascr$ and $\Bscr_i$ to have as
    idempotents only $0$ and $1$. On the other hand, the Morita equivalence guarantees that
    $\Ascr^{\mathrm{op}}\iso{_{\Bscr_i}\Hom}(M_i,M_i)$ for some projective left $\Bscr_i$-module $M_i$. Since
    $\Ascr^{\mathrm{op}}\iso{_{\Ascr}\Hom}(\Ascr,\Ascr)$, it follows that the uniqueness
    part of the Wedderburn property fails for $R$.
\end{example}

The impetus to think about these questions came from a conversation of the first-named author with Colin Ingalls, Daniel Krashen, and David
Saltman on a hike to Emmaline Lake in Pingree Park, Colorado during the 10th Brauer Group Conference in August 2012. We thank Lawrence Ein
who mentioned to the same author the affine Lefschetz hyperplane theorem. Finally, we thank the referee for many suggestions, which greatly
improved the exposition.

\section{Proof}

The proof is topological in character.

An $n$-equivalence is defined to be a map $h: X \to Y$ of topological spaces such that $\pi_i(h)$ is an isomorphism for $i<n$ and a surjection for $i=n$ for all
choices of basepoint. Recall that in the construction of Postnikov towers of pointed spaces~\cite{hatcher}*{Chapter 4}, for any pointed
space $(X,x)$ there is a natural $(n+1)$-equivalence $(X,x)\rightarrow(\tau_{\leq n}X,x)$, where
\begin{equation*}
    \pi_i(\tau_{\leq n}X,x)\iso\begin{cases}
        \pi_i(X,x)  &   \text{if $i\leq n$,}\\
        0       &   \text{if $i>n$.}
    \end{cases}
\end{equation*}
In the remainder of the paper we omit the basepoint from the notation. We shall also assume that all constructions carried out on
topological spaces result in spaces having the homotopy-type of CW complexes.

Throughout, $\SL_n$ shall be used for $\SL_n(\CC)$, and similarly for other classical
groups. If $m$ and $n$ are positive integers such that $m$ divides $n$, write $\P(m,n)$ for the quotient of the special linear group $\SL_n$ by
the central subgroup $\mu_m$ of $m$-th roots of unity. Note the special cases $\P(1,n)=\SL_n$ and $\P(n,n) = \PGL_n$. We will freely make
use of the homotopy equivalences $\SU_n \to \SL_n$ and $\SU_n/ \mu_m \to
\P(m,n)$, where $\SU_n$ denotes the special unitary group.

In topology, the cohomological Brauer group of a space is $\Hoh^3(X,\ZZ)_{\tors}$. This and further background on topological Azumaya
algebras and the topological Brauer group may be found in~\cite{aw1}. We recall that $\BPGL_n$ classifies degree-$n$ Azumaya algebras. There is
a natural map from $\BPGL_n$ to the Eilenberg-MacLane space $K(\ZZ/n,2)$ the composition of which with the Bockstein $K(\ZZ/n,2)\rightarrow
K(\ZZ,3)$ yields the Brauer class of the degree-$n$ Azumaya algebra.

We shall make use throughout of the following elementary calculations:
\[ \Hoh^1(\BP(m,n), \ZZ) = \Hoh^2(\BP(m,n), \ZZ) = 0, \qquad \Hoh^3(\BP(m,n), \ZZ) = \ZZ/m. \]

 There is a homotopy-pullback diagram
\begin{equation*} \label{eq:mapsNeeded}
\xymatrix{\BP(m,n) \ar[r] \ar[d] &    \BPGL_n \ar[d] \\
        K(\ZZ/m,2)  \ar[r] &   K(\ZZ/n,2),}
\end{equation*}
where the top horizontal arrow is induced by the quotient map $\P(m,n)\rightarrow\PGL_n$, and the bottom horizontal arrow is induced by the
inclusion of $\ZZ/m$ into $\ZZ/n$.  The space $\BP(m,n)$ is equipped with a canonical
degree-$n$ topological Azumaya algebra $\Ascr$ such that the
class $\cl(\Ascr)$ is $m$-torsion in the topological Brauer group.

To show that a class in the topological Brauer group may have period $2$ and index $2$, while not being represented
by a topological Azumaya algebra of degree $2$, we construct an explicit example of such a class $\cl(\Ascr)$, where $\Ascr$ is of degree $2n$ for $n >
1$ odd. To do this, we shall exhibit certain spaces $X$ with the additional property that
\[ \Hoh^1(X, \ZZ) = \Hoh^2(X, \ZZ) = 0, \qquad \Hoh^3(X, \ZZ) = \ZZ/2. \] 
The Azumaya algebra $\Ascr$ will have Brauer class $\cl(\Ascr)$
generating $\Hoh^3(X,\ZZ)$. Since $\Hoh^2(X, \ZZ) = 0$, and since $\Hoh^2(\BPGL_2, \ZZ) = 0$, there is a bijection between $2$--torsion
cohomology classes $X \to K(\ZZ,3)$ and cohomology classes $X \to K(\ZZ/2, 2)$ and similarly for $\BPGL_2$. It will suffice to show that the
map $X \to K(\ZZ/2, 2)$ cannot be factored as $X \to \BPGL_2 \to K(\ZZ/2,2)$, which is equivalent to showing that there is no map $X \to
\BPGL_2$ inducing an isomorphism on $\Hoh^2(\cdot, \ZZ/2)$. We shall show below that any space $X$ which is $6$--equivalent to $\BP(2,2n)$
with $n>1$ odd will serve. The space $\BP(2,2n)$ itself settles the topological version of the question, but is not a variety; in order to apply
the result in algebraic geometry we shall have to use a finite approximation to $\BP(2,2n)$.

The obstruction we arrive at is in the higher homotopy groups and we collect some relevant facts regarding these here. The group $\SU_2$ is the
group of unit quaternions, and is homeomorphic to $S^3$. The group $\PU_2$ is homeomorphic to $\RP^3$, and the projection map $S^3 \to
\RP^3$, equivalently $\SU_2 \to \PU_2$, induces an isomorphism on all homotopy groups except $\pi_1$, where $\pi_1(\PU_2) = \ZZ/2$. Using $\PU_2\we\PGL_2$, the
classifying-space functor, and the fact that $\pi_4(S^3)=\ZZ/2$ (see~\cite{hatcher}*{Corollary 4J.4}) gives:
\begin{align*}
  \pi_2(\BPGL_2) & = \ZZ/2, & & \pi_4(\BPGL_2)  = \ZZ, \\
  \pi_3(\BPGL_2) & = 0, & & \pi_5(\BPGL_2)  =  \ZZ/2.
\end{align*}
If $n>1$, then Bott periodicity shows that:
\begin{align*}
  \pi_2(\BP(2,2n)) & = \ZZ/2, & & \pi_4(\BP(2,2n))  = \ZZ, \\
  \pi_3(\BP(2,2n)) & = 0, & & \pi_5(\BP(2,2n))  =  0.
\end{align*}
There is a map $\SL_2 \to \SL_{2n}$ given by $n$-fold block-summation, which descends to a map $\PGL_2 \to
\P(2,2n)$. The induced map $\pi_2 (\BPGL_2) \to \pi_2(\BP(2,2n))$ is an isomorphism, and therefore so too is the map $\Hoh^2(\BP(2,2n),
\ZZ/2) \to \Hoh^2(\BPGL_2, \ZZ/2) \iso \ZZ/2$.

The main technical lemma of this paper is the following:
\begin{lemma} \label{l:1}
    Suppose that $f: \tau_{\leq 5}\BPGL_2\to \tau_{\leq 5} \BPGL_2$ is a map which
  induces an isomorphism on $\Hoh^2(\tau_{\leq 5}\BPGL_2, \ZZ/2) \iso \ZZ/2$. Then $f$ induces an isomorphism on $\pi_5(\tau_{\leq 5} \BPGL_2) \iso \ZZ/2$.
\end{lemma}
\begin{proof}
  It suffices to show that if $\Omega f : \Omega \tau_{\leq 5} \BPGL_2 \to \Omega \tau_{\leq 5} \BPGL_2$ induces an isomorphism on $\Hoh^1(\Omega
  \tau_{\leq 5} \BPGL_2, \ZZ/2) = \Hoh^2(\tau_{\leq 5} \BPU_2, \ZZ/2)$, then $\Omega f$ induces an isomorphism on $\pi_4( \Omega \tau_{\leq 5} \BPGL_2) =
  \pi_5(\tau_{\leq 5} \BPGL_2)$.  Note that $\Omega  \tau_{\leq 5} \BPGL_2 \weq \tau_{\leq 4}\PGL_2$, by consideration of homotopy groups.

  We proceed in three steps to complete the proof. We first show that $\Omega f$ induces an isomorphism on $\Hoh_3(\tau_{\leq 4}\PGL_2, \ZZ/2)$, then
  on $\pi_3(\tau_{\leq 4}\PGL_2) \tensor_\ZZ \ZZ/2$, and lastly on $\pi_4( \tau_{\leq 4}\PGL_2)$.  

  The map $\RP^3\we\PGL_2\rightarrow\tau_{\leq 4}\PGL_2$ is a $5$-equivalence.
  Therefore, $\Hoh^{\le 4}(\tau_{\leq 4}\PGL_2, \ZZ/2)$ is the ring $\ZZ[x]/(2,x^4)$ with $\deg x =1$. The map $\Omega f$ induces an
  isomorphism on the vector space $\Hoh^3(\tau_{\leq 4}\PGL_2, \ZZ/2)$, and so on the dual space $\Hoh_3( \tau_{\leq 4}\PGL_2, \ZZ/2)$. This completes the first step.

  The first step implies that the map induced on $\Hoh_3(\tau_{\leq 4} \PGL_2, \ZZ) \iso \ZZ$ is multiplication by an odd integer,
  $q$. The Hurewicz map is a natural transformation of functors, and so there is a diagram
\[ \xymatrix@C=80pt{ \pi_3(S^3) \ar^{\iso}[r] \ar^{\iso}[d] & \pi_3(\PGL_2) \iso \pi_3(\tau_{\leq 4} \PGL_2) \ar@{^(->}^{\times 2}[d] \\
\Hoh_3(S^3,\ZZ) \ar^{\times 2}@{^(->}[r]  & \Hoh_3(\PGL_2, \ZZ) \iso \Hoh_3(\tau_4 \PGL_2, \ZZ). } \]
We can identify $\pi_3(\tau_{\le 4} \PGL_2)$ with the index-$2$ subgroup of $\Hoh_3(\tau_{\le 4} \PGL_2)\iso \ZZ$. The map $\Omega f$ therefore
induces multiplication by an odd integer, $q$, on $\pi_3(\tau_{\le 4} \PGL_2)$, and consequently an isomorphism on $\pi_3(\tau_{\le 4}
\PGL_2) \tensor_{\ZZ} \ZZ/2$. This completes the second step.

  Finally, there is a natural transformation $(\Sigma\eta)^*: \pi_3( X ) \to \pi_4( X )$ given by precomposition with the suspension of the
  Hopf map $\Sigma\eta: S^4 \to S^3$. One can verify that this is in fact a natural
  homomorphism of groups, see \cite{whitehead}*{X(8)}.  Since there
  is an isomorphism $\pi_3(S^3) \to \pi_3(\tau_{\leq 4}\PGL_2)$, and since the natural transformation $(\Sigma \eta)^*$ induces an isomorphism $\pi_3(
  S^3) \tensor_\ZZ \ZZ/2 \iso \pi_4(S^3)$, it follows that there is a natural isomorphism $\pi_3( \tau_{\leq 4}\PGL_2 ) \tensor_\ZZ \ZZ/2 \iso
  \pi_4(\tau_{\leq 4}\PGL_2)$. By naturality, $\Omega f$ therefore induces an isomorphism on $\pi_4( \tau_{\leq 4}\PGL_2)$.
\end{proof}

\begin{proposition}\label{prop:1}
  Let $n>1$ be an integer. Suppose $X$ is a CW complex and $h: X \to \tau_{\leq 5} \B\P(2,2n)$ is a $6$--equivalence. There is no map $f: X
  \to \BPGL_2$ inducing an isomorphism on $\Hoh^2(\cdot, \ZZ/2)$.
    \begin{proof}
      Suppose for the sake of contradiction that such a map $f$ exists. Let $s$ denote a homotopy-inverse to the equivalence
      $\tau_{\leq 5}(h)$. Then the composite
        \begin{equation*}
            \xymatrix{\tau_{\leq 5} \BPGL_2 \ar[r] & \tau_{\leq 5} \B \P(2,2n) \ar^{\hspace{1.2em}s}[r] & \tau_{\leq 5} X\ar^{\tau_{\leq 5} f}[rr] & & \tau_{\leq 5} \BPGL_2,}
        \end{equation*}
        where the first map is given by block-summation, induces an isomorphism on the cohomology group $\Hoh^2(\tau_{\leq 5}\BPGL_2, \ZZ/2)$. Since
        $\pi_5(\BP(2,2n))=0$, the composite is necessarily the $0$--map on $\pi_5(\tau_{\leq 5}\BPGL_2)$, contradicting Lemma \ref{l:1}.
    \end{proof}
\end{proposition}

We are now ready to prove our main theorem.  If $X$ is a scheme and $\alpha\in\Br(X)$, we let $\ind(\alpha)$ denote the greatest common
divisor of the degrees of all Azumaya algebras in the class $\alpha$. When $X$ is a regular and noetherian integral scheme, we showed
in~\cite{aw3}*{Proposition 6.1}, using an argument suggested by Saltman, that $\ind(\alpha)=\ind(\alpha_\eta)$, where $\eta$ is the generic
point of $X$.  Note that for regular noetherian schemes, if $n$ is odd and $\Ascr$ is an Azumaya algebra of degree $2n$ and period $2$, then
the Brauer class $\cl(\Ascr)$ has index $2$, because the period and index have the same prime divisors by~\cite{aw3}*{Proposition 6.1}.

\begin{proof}[Proof of Theorem 1.1]
  Let $V$ be an algebraic linear representation of $\P(2,2n)$ over $\CC$ such that $\P(2,2n)$ acts freely outside an invariant closed
  subscheme $S$ of codimension at least $4$, and such that $Y=(V-S)/\P(2,2n)$ exists as a smooth quasi-projective complex variety. Such
  representations exist by~\cite{totaro}*{Remark 1.4}. There is a classifying map $Y \to \BP(2,2n)$, classifying the canonical algebraic
  $\P(2,2n)$--torsor $V-S \to (V-S)/\P(2,2n)$. As the codimension of $S$ is at least $4$, the scheme $V-S$ is $6$--connected.  It follows
  from a map of long exact sequences of homotopy groups that $Y \to \BP(2,2n)$ is a $7$--equivalence of topological spaces, and so there
  exists $6$--equivalence $Y \to \tau_{\le 5} \BP(2, 2n)$.

  The algebraic $\P(2,2n)$-torsor on $Y$ induces a canonical algebraic Azumaya algebra $\Ascr$ over $Y$ of degree $2n$ and of period
  $2$. Suppose $\Bscr$ is an algebraic Azumaya algebra over $Y$ of (algebraic) period $2$. Then the class of $\Bscr$ in the algebraic
  cohomological Brauer group $\Hoh^2_\et(Y, \Gm)_{\tors}$ is of order $2$. Since $\Hoh^2_\et(Y,
  \mu_2) \iso \Hoh^2(Y, \ZZ/2) = \ZZ/2$, it follows that there is a unique lift of
  $\cl(\Bscr)$ to $\Hoh^2_{\et}(Y,\ZZ/2)$. By Proposition~\ref{prop:1}, there is no map $Y \to \BPGL_2$ which is nontrivial on
  $\Hoh^2(\cdot, \ZZ/2)$. Therefore, $\Bscr$ is not of degree $2$. In particular, $\Ascr$ is not
  equivalent to any Azumaya algebra of degree $2$.

  Although the variety $Y$ as constructed need not be affine, the argument as carried out above relies only on the $6$-equivalence $Y
  \to \tau_{\le 5} \BP(2,2n)$ and the algebraic nature of the canonical degree-$6$ Azumaya
  algebra. The variety $Y$ may be replaced,
  using Jouanolou's device~\cite{jouanolou}, by an affine bundle $p: X \to Y$
  such that $X$ is smooth and affine. The map $p$ is a homotopy equivalence, since the
  fibers are affine spaces. One may pull-back the Azumaya algebra $\Ascr$ to $X$, and there is no degree $2$ topological
  Azumaya algebra that is equivalent to $p^* \Ascr$.

  Once such an $X$ has been found, one may employ the affine Lefschetz hyperplane theorem (see~\cite{goresky-macpherson}*{Introduction,
    Section 2.2}), which says that if $H$ is a generic hyperplane in $\A^k$, then $X\cap H\rightarrow X$ is a $(j-1)$-equivalence. By
  intersecting many times, all the while ensuring that the intersections are smooth, we can replace $X$ by a $6$-dimensional smooth affine
  variety.
\end{proof}

We remark that our method of proof is to show that there exists a finite CW complex $X$ and a class $\alpha\in\Br(X)\iso\Hoh^3(X,\ZZ)_{\tors}$ such that
$\ind(\alpha)=2$, but where $\alpha$ is not represented by a degree $2$-topological Azumaya algebra. This is a new result even in the setting of
topological Azumaya algebras studied in~\cite{aw1}.

\end{document}